\documentclass[11pt]{amsart}

\usepackage{amscd,amssymb,amsopn,amsmath,amsthm,mathrsfs,graphics,amsfonts,enumerate,verbatim,calc
}
\usepackage{bbm}
\usepackage[all,cmtip]{xy}
\usepackage[all]{xy}
\usepackage{tikz}
\usepackage{lscape}
\usepackage{enumitem}
\usetikzlibrary{matrix,arrows,decorations.pathmorphing,cd}
\usepackage{hyperref}
\hypersetup{colorlinks=true,linkcolor={blue}}
\usepackage{cleveref}

\usepackage{scalerel}
\usepackage{stackengine,wasysym}
\usetikzlibrary {positioning}
\usetikzlibrary{patterns}
\usetikzlibrary{calc}
\definecolor {processblue}{cmyk}{0.96,0,0,0}

\usepackage{subfigure}

\usepackage{comment} 

\usepackage{ dsfont }

\usepackage{color}

\usepackage[OT2,OT1]{fontenc}
\newcommand\cyr{%
\renewcommand\rmdefault{wncyr}%
\renewcommand\sfdefault{wncyss}%
\renewcommand\encodingdefault{OT2}%
\normalfont
\selectfont}
\DeclareTextFontCommand{\textcyr}{\cyr}

\usepackage{amssymb,amsmath}

\DeclareFontFamily{OT1}{rsfs}{}
\DeclareFontShape{OT1}{rsfs}{n}{it}{<-> rsfs10}{}
\DeclareMathAlphabet{\mathscr}{OT1}{rsfs}{n}{it}

\numberwithin{equation}{section}
\newtheorem{theorem}{Theorem}[section]
\newtheorem{lem}[theorem]{Lemma}
\newtheorem{prop}[theorem]{Proposition}

\theoremstyle{definition}

\theoremstyle{remark}
\newtheorem{remark}[theorem]{Remark}

\newcommand{\Ass}{\operatorname{Ass}}

\renewcommand{\ker}{\operatorname{ker}}

\newcommand{\Spec}{\operatorname{Spec}}

\newcommand{\id}{\operatorname{id}}
\newcommand{\Ext}{\operatorname{Ext}}
\newcommand{\tr}{\operatorname{tr}}

\newcommand{\Supp}{\operatorname{Supp}}
\newcommand{\Tor}{\operatorname{Tor}}
\newcommand{\Hom}{\operatorname{Hom}}
\newcommand{\End}{\operatorname{End}}

\newcommand{\coker}{\operatorname{coker}}

\newcommand{\Ann}{\ensuremath{\operatorname{Ann}}}

\newcommand{\p}{\mathfrak{p}}

\newcommand{\m}{\mathfrak{m}}

\newcommand{\w}{\omega}

\begin{document}
\title[Centers of Endomorphism Rings and Reflexivity]{Centers of Endomorphism Rings and Reflexivity}

\author[Dey]{Souvik Dey}
\address{Department of Mathematical Sciences\\
University of Arkansas\\
Fayetteville, AR 72701 USA}
\email[Souvik Dey]{souvikd@uark.edu}
\urladdr{https://sites.google.com/view/souvikdey}

\author[Lyle]{Justin Lyle}
\address{Department of Mathematics and Statistics \\ 221 Parker Hall\\
Auburn University\\
Auburn, AL 36849}
\email[Justin Lyle]{jll0107@auburn.edu}
\urladdr{https://jlyle42.github.io/justinlyle/}

\thanks{Souvik Dey was partly supported by the Charles University Research Center program No.UNCE/SCI/022 and a grant GACR 23-05148S from the Czech Science Foundation}



\begin{abstract}
Let $R$ be a local ring and let $M$ be a finitely generated $R$-module. Appealing to the natural left module structure of $M$ over its endomorphism ring and corresponding center $Z(\End_R(M))$, we study when various homological properties of $M$ are sufficient to force $M$ to have a nonzero free summand. Consequences of our work include a partial converse to a well-known result of Lindo describing $Z(\End_R(M))$ when $M$ is faithful and reflexive, as well as some applications to the famous Huneke-Wiegand conjecture.
\end{abstract}

\maketitle

\section{Introduction}


Let $R$ be a commutative Noetherian local ring and let $M$ be a finitely generated $R$-module. The endomorphism algebra $\End_R(M)$ captures a wealth of information about the module $M$, and even about the ring $R$ itself. However, this information can be difficult to access in general; the algebra $\End_R(M)$ is nearly always noncommutative, and nearly always much larger in several senses than the ring $R$. Narrowing our attention to the center $Z(\End_R(M))$ of $\End_R(M)$, we retain a wealth of desirable information, but are left with an algebra that much easier to work with in general, and obviously more amenable to the techniques of commutative algebra. Indeed, work of Lindo in the previous decade shows there is a tight connection between $Z(\End_R(M))$ and endomorphism algebras of trace ideals, classically studied objects that work of Lindo and others have spurred into renewed attention; see e.g. \cite{Li17,HH19,Fa20,HK23,DL24,LM24} for but a few examples. 

To be concrete, the \emph{trace ideal} of $M$ is the ideal
\[\tr_R(M):=\langle f(x) \mid f \in \Hom_R(M,R), x \in M \rangle\]
and the key result of Lindo is the following:
\begin{theorem}[{\cite[Theorem 3.9]{Li17}}]\label{Lindointro}
Suppose $M$ is a faithful reflexive $R$-module. Then there is an isomorphism of $R$-algebras $Z(\End_R(M)) \cong \End_R(\tr_R(M))$.
\end{theorem}

The goal of this note is to expound upon the isomorphism in Theorem \ref{Lindointro} and to apply these ideas to provide sufficient criteria for when $Z(\End_R(M)) \cong R$. 

Our first main result offers a partial converse to Lindo's result:
\begin{theorem}\label{mainthm1intro}(see \Cref{Lindoconverse})

Let $R$ be a local ring and let $M$ be a finitely generated $R$-module. Let $(-)^*:=\Hom_R(-,R)$.

\begin{enumerate}

\item[$(1)$] Suppose $\Supp(M)=\Spec(R)$. If $\Ext^1_R(M,M)=0$ and $\End_R(M) \cong \Hom_R(M,M^{**})$ as right $\End_R(M)$-modules, then $M$ is reflexive.

\item[$(2)$] Suppose $M$ is torsion-free and is locally free of full support on the punctured spectrum of $R$. Further suppose that $\End_R(M)$ is local, e.g. $R$ is Henselian and $M$ is indecomposable. If $\Ext^1_{\End_R(M)}(M,M)=0$ and $Z(\End_R(M)) \cong \End_R(\tr_R(M))$, then $M$ is reflexive.

\end{enumerate}

\end{theorem}

A byproduct of Theorem \ref{Lindointro} is that $Z(\End_R(M))$ is a birational extension of $R$ under mild hypotheses, and a sufficient understanding of the properties of this extension would offer solutions to several longstanding conjectures, e.g. the Huneke-Wiegand conjecture and certain cases of the Auslander-Reiten conjecture. For instance, Lindo uses these ideas to prove a torsion-free faithful module $M$ must have a free summand if $\Ext^1_R(M,M)=0$ and the rings $R$ and $Z(\End_R(M))$ are Gorenstein \cite[Theorem 6.4]{Li17}. It can be drawn from this that the Huneke-Wiegand conjecture will hold for a Gorenstein domain $R$ of dimension $1$ if the condition $\Ext^1_R(M,M)=0$ always implies that $Z(\End_R(M))$ is Gorenstein. Indeed, the condition $\Ext^1_R(M,M)=0$ is preserved under removing free summands, so Lindo's result can be applied inductively.

Our second main theorem adopts this perspective to provide several similar sufficient homological criteria under which $M$ must have a nonzero free summand. 

\begin{theorem}\label{mainthm2intro}(see Theorems \ref{thm4.2} and \ref{thm4.4})
Let $R$ be Cohen-Macaulay local ring of dimension $1$ and let $M$ be a finitely generated torsion-free $R$-module of positive rank satisfying $\Ext^1_R(M,M)=0$. Suppose additionally that one of the following holds: 
\begin{enumerate}
    \item[$(1)$] $M$ is reflexive as a $Z(\End_R(M))$-module.
    \item[$(2)$] $\Ext^2_{Z(\End_R(M))}(M,M)=0$.
    \item[$(3)$] $M \cong I$ where $I$ is an ideal of positive grade in $R$ and $\Ext^1_{\End_R(I)}(I,\End_R(I))=0$,
    
\end{enumerate}
Then $Z(\End_R(M)) \cong R$. If additionally, $M$ is a reflexive $R$-module, then $R$ is a summand of $M$.

\end{theorem}

In particular, as any torsion-free module over a Gorenstein ring of dimension $1$ is reflexive, Theorem \ref{mainthm2intro} gives a direct extension of \cite[Theorem 6.4]{Li17}, and shows that passage of reflexivity from the base $R$ to $Z(\End_R(M))$ can be viewed as a cornerstone of the Huneke-Wiegand conjecture.

We now outline the structure of our paper. In Section \ref{prelims} we provide some background and some preliminary results needed throughout. In Section \ref{lindoconversesection}, we prove Theorem \ref{mainthm1intro} and comment on its sharpness. Section \ref{rigidsection} gives several partial answers to the question as to when $Z(\End_R(M)) \cong R$, ultimately giving a proof of Theorem \ref{mainthm2intro}.

\section{Preliminaries}\label{prelims}

Throughout, we let $(R,\m,k)$ be a commutative Noetherian local ring. All $R$-modules are assumed to be finitely generated unless stated otherwise. If $M$ is an $R$-module, we let $\Omega^i_R(M)$ denote the $i$th syzygy of $M$. Throughout, we set $E:=\End_R(M)$ and note $M$ is natural a left $E$-module where $f \cdot x=f(x)$ for $f \in E$ and $x \in M$, and so $M$ in particular inherits a $Z(E)$-module structure. We let $(-)^*:=\Hom_R(-,R)$ and when $R$ is Cohen-Macaulay with canonical module $\w_R$, we let $(-)^{\vee}:=\Hom_R(-,\w_R)$. We recall that $M$ is said to be a \emph{generator} if for all $R$-modules $N$, there is a surjection $M^{\oplus a_N} \to N$ for some $a_N$. As $R$ is local, it is easily seen that $M$ is a generator if and only if $R$ is a direct summand of $M$.

A key object of study in this work is the trace submodule $\tr_N(M)$ of $M$ in $N$ defined as 
\[\tr_N(M):=\langle f(x) \mid f \in \Hom_R(M,N), x \in M \rangle.\] 
In the special case where $N=R$, $\tr_R(M)$ is an ideal of $R$ referred to simply as the \emph{trace ideal} of $M$. While one can see \cite{Li17,Fa20,Ly23} for a comprehensive treatment of trace submodules and trace ideals, we recall some of the key properties below that are suited to our needs:
\begin{prop}[see {\cite[Proposition 3.1]{Ly23}}]\label{basictraceprop}
If $A,B$ and $N$ are $R$-modules and $I$ is an ideal, then
\begin{enumerate}
\item[$(1)$] $\tr_N(A \oplus B)=\tr_N(A)+\tr_N(B)$.
\item[$(2)$] If $i:\tr_N(A) \to N$ is the natural inclusion, then $\Hom_R(A,i):\Hom_R(A,\tr_N(A)) \to \Hom_R(A,N)$ is an isomorphism.
\item[$(3)$] $M$ is a generator if and only if $\tr_R(M)=R$.
\item[$(4)$] If $L$ is a submodule of $N$, then $L \subseteq \tr_N(L)$ with equality if and only if $L=\tr_R(M)$ for some $R$-module $M$.
\item[$(5)$] If $S$ is a flat $R$-algebra, then $\tr_{N}(A \otimes_R S)=\tr_N(A)S$. In particular, taking trace respects localization and completion.

\end{enumerate}

\end{prop}

We also record several facts contrasting modules structures over $R$ versus those over $Z(E)$ or $E$ that are generally known to experts:

\begin{prop}\label{RvsSprop}

Suppose $M$ is a torsionless faithful $R$-module.
\begin{enumerate}
\item[$(1)$] If $A$ is a right $E$-module and $B$ is a left $E$-module, then the natural map $A \otimes_R B \to A \otimes_E B$ is a surjection of $Z(E)$-modules. 
\item[$(2)$] If $A$ and $B$ are $Z(E)$-modules, then the natural map $p:A \otimes_R B \to A \otimes_{Z(E)} B$ is a surjection of $Z(E)$-modules with $\ker(p)$ torsion.
\item[$(3)$] If $M$ is torsionless, and if $A$ and $B$ are $Z(E)$-modules with $B$ torsion-free, then $\Hom_R(A,B)=\Hom_{Z(E)}(A,B)$.
\end{enumerate}

\end{prop}

\begin{proof}

Since $M$ is faithful, $R$ naturally embeds as a subring of both $Z(E)$ and $E$, so the maps $A \otimes_R B \to A \otimes_E B$ and $A \otimes_R B \to A \otimes_{Z(E)} B$ are surjective. As $M$ is torsionless and faithful, it follows from \cite[Corollary 3.8]{Li17} that $Z(E)$ is a birational extension of $R$. In particular, $Z(E)_{\p} \cong R_{\p}$ for all $\p \in \Ass(R)$. Thus $p$ is locally a surjective endomorphism and thus isomorphism at every associated prime of $R$, from which we see that $\ker(p)$ is torsion. 

For $(3)$, we note as above that $Z(E)$ is a birational extension of $R$. Then the claim follows immediately from \cite[Proposition 4.14]{Lw12}.

\end{proof}


The next proposition compares canonical duality of $M$ over $R$ versus that over $Z(E)$.

\begin{prop}\label{samedual}
Suppose $R$ is a Cohen-Macaulay local ring of dimension $1$ with canonical module $\w_R$. Suppose $M$ is a torsion-free $R$-module, and suppose that either $M$ is generically free or that $R$ is generically Gorenstein. Then $M^{\vee} \cong \Hom_{Z(E)}(M,\w_{Z(E)})$ as $Z(E)$-modules.
\end{prop}

\begin{proof}
It follows from \cite[Corollary 3.6]{Ly23} that $\w_{Z(E)} \cong \tr_{\w}(M)$. The claim then follows from combining Proposition \ref{basictraceprop} (2) Proposition \ref{RvsSprop} (3).
\end{proof}

Proposition \ref{samedual} implies there is no ambiguity in its setting between canonical dual of $M$ viewed as an $R$-module or viewed as a $Z(E)$-module. We will thus identify them in this setting for the rest of this work without further comment.



\section{A Partial Converse to A Result of Lindo}\label{lindoconversesection}

We will need a very general version of a duality map for the purposes of this section: If $\Lambda$ is a possibly noncommutative ring and $M,L,N$ are left $\Lambda$-modules, we let \[\beta_{L,N}^{M,\Lambda}:\Hom_{\Lambda}(L,N) \to \Hom_{\End_\Lambda(M)}(\Hom_{\Lambda}(M,L),\Hom_{\Lambda}(M,N))\] be the map of $\End_R(N)-\End_R(L)$ bimodules given by $f \mapsto \Hom(M,f)$.

We will make use of the following well-known lemma:

\begin{lem}\label{quicklem}
If $\Lambda$ is a possibly noncommutative ring, then for any left $\Lambda$-modules $L,N$, the map $\beta^{L,\Lambda}_{L,N}$ is an isomorphism.
\end{lem}

\begin{proof}

Define $\gamma:\Hom_{\End_{\Lambda}(L)}(\End_{\Lambda}(L),\Hom_{\Lambda}(L,N)) \to \Hom_{\Lambda}(L,N)$ by $g \mapsto g(1_{\End_{\Lambda}(L)})$. We observe for any $f \in \Hom_{\End_{\Lambda}(L)}(\End_{\Lambda}(L),\Hom_{\Lambda}(L,N))$ that $\beta^{L,\Lambda}_{L,N}(\gamma(f))=\Hom_{\Lambda}(L,f(1_{\End_{\Lambda}(L)}))$. But if $x \in \End_{\Lambda}(L)$, then $\Hom_{\Lambda}(L,f(1_{\End_{\Lambda}(L)}))(x)=f(1_{\End_{\Lambda}(L)}) \circ x=f(1_{\End_{\Lambda}(L)} \circ x)=f(x)$. That is to say, $\Hom_{\Lambda}(L,f(1_{\End_{\Lambda}(L)}))=f$, so $\beta^{L,\Lambda}_{L,N} \circ \gamma=\id_{\Hom_{\End_{\Lambda}(L)}(\End_{\Lambda}(L),\Hom_{\Lambda}(L,N))}$. 

Similarly, for any $g \in \Hom_{\Lambda}(L,N)$, we have $\gamma(\beta^{L,\Lambda}_{L,N}(g))=\gamma(\Hom_{\Lambda}(L,g))=\Hom_{\Lambda}(L,g)(1_{\End_{\Lambda}(L)})=g \circ 1_{\End_{\Lambda}(L)}=g$. Thus $\gamma \circ \beta^{L,\Lambda}_{L,N}=\id_{\Hom_{\Lambda}(L,N)}$. So $\beta^{L,\Lambda}_{L,N}$ and $\gamma$ are inverses, and in particular $\beta^{L,\Lambda}_{L,N}$ is an isomorphism.

\end{proof}

The following gives a partial converse to Theorem \ref{Lindointro}:

\begin{theorem}\label{Lindoconverse}

Let $M$ be an $R$-module.

\begin{enumerate}

\item[$(1)$] Suppose $M$ has full support. 
If $\Ext^1_R(M,M)=0$ and $E \cong \Hom_R(M,M^{**})$ as right $E$-modules, then $M$ is reflexive. 

\item[$(2)$]  Suppose $M$ is torsion-free and is free of full support on the punctured spectrum of $R$. Further suppose that $\End_R(M)$ is local, e.g. $R$ is Henselian and $M$ is indecomposable. If $\Ext^1_E(M,M)=0$ and $Z(E) \cong Z(\End_R(M^*))$ as $Z(E)$-modules, then $M$ is reflexive.

\end{enumerate}

\end{theorem}

\begin{proof}

We begin with $(1)$. Suppose $E$ and $\Hom_R(M,M^{**})$ are isomorphic as right $E$-modules, i.e., that there is an isomorphism $\gamma \in \Hom_E(E,\Hom_R(M,M^{**})$ By Lemma \ref{quicklem}, $\beta^{M,R}_{M,M^{**}}$ is an isomorphism, so there is a map of $R$-modules $\alpha:M \to M^{**}$ so that $\gamma=\Hom_R(M,\alpha)$.

Applying $\Hom_R(M,-)$ to the exact sequence 
\[0 \rightarrow \ker(\alpha) \rightarrow M \xrightarrow{\alpha} M^{**}\]
we get an exact sequence 
\[0 \rightarrow \Hom_R(M,\ker(\alpha)) \rightarrow E \xrightarrow{\gamma} \Hom_R(M,M^{**}).\]
Since $\gamma$ is an isomorphism, it follows that $\Hom_R(M,\ker(\alpha))=0$. Thus there is a $\ker(\alpha)$ regular element $x \in \Ann_R(M)$. But $\ker(\alpha) \hookrightarrow M$, so if $xM=0$, then $x\ker(\alpha)=0$ as well. It follows that $\ker(\alpha)=0$. 

Then applying $\Hom_R(M,-)$ to the exact sequence
\[0 \rightarrow M \xrightarrow{\alpha} M^{**} \rightarrow \coker(\alpha) \rightarrow 0\]
we obtain, since $\Ext^1_R(M,M)=0$, an exact sequence 
\[0 \rightarrow E \xrightarrow{\gamma} \Hom_R(M,M^{**}) \rightarrow \Hom_R(M,\coker(\alpha)) \rightarrow 0.\]
Since $\gamma$ is an isomorphism, we have $\Hom_R(M,\coker(\alpha))=0$. But then $\Ass_R \Hom_R(M,\coker(\alpha))=\Supp_R(M) \cap \Ass_R(\coker(\alpha))=\Spec(R) \cap \Ass_R(\coker(\alpha))=\Ass(\coker(\alpha))=\varnothing$, which implies $\coker(\alpha)=0$. 
Therefore, $\alpha$ is an isomorphism, so $M$ is reflexive.

For $(2)$, we follow a similar approach as for $(1$). We first note that is follows from \cite[Theorem 3.21]{Li17} that $Z(\End_R(M^*)) \cong \End_R(\tr_R(M))$, while \cite[Theorem 3.3 (1)]{Ly23} 
implies that $\End_R(\tr_R(M)) \cong \End_{E}(M^*)$. From Hom-tensor adjointness we see that $\End_{E}(M^*) \cong \Hom_E(M,M^{**})$. So there is an isomorphism of $Z(E)$-modules, $\eta:Z(E) \to \Hom_E(M,M^{**})$. By Lemma \ref{quicklem}, $\beta^{M,E}_{M,N}$ is an isomorphism. Thus there is a map $\theta:M \to M^{**}$ of left $E$-modules so that $\eta=\Hom_E(M,\theta)$. Applying $\Hom_E(M,-)$ to the exact sequence
\[0 \rightarrow \ker(\theta) \rightarrow M \xrightarrow{\theta} M^{**},\]
we get an exact sequence
\[0 \rightarrow \Hom_E(M,\ker(\theta)) \rightarrow Z(E) \xrightarrow{\eta} \Hom_E(M,M^{**}).\]
As $\eta$ is an isomorphism, it follows that $\Hom_E(M,\ker(\theta))=0$. As $M$ is free on the punctured spectrum of $R$, localizing at any $\p \in \Spec(R)-\{\m\}$, we have $M_{\p} \cong R_{\p}^{\oplus r_{\p}}$, and we see that 
\[\Hom_E(M,\ker(\theta))_{\p} \cong \Hom_{E_{\p}}(M_{\p},\ker(\theta)_{\p}) \cong \Hom_{M_{r_{\p}}(R_{\p})}(R_{\p}^{\oplus r_{\p}},\ker(\theta)_{\p})=0.\]
As there is an isomorphism of left $M_{r_{\p}}(R_{\p})$-modules $M_{r_{\p}}(R_{\p}) \cong (R_{\p}^{\oplus r_{\p}})^{\oplus r_{\p}}$, it follows that $\ker(\theta)_{\p}=0$. It follows that $\ker(\theta)$ is an $R$-module of finite length, but as $M$ is torsion-free, it must be that $\ker(\theta)=0$. Applying $\Hom_E(M,-)$ to the exact sequence 
\[0 \rightarrow M \xrightarrow{\theta} M^{**} \rightarrow \coker(\theta) \rightarrow 0\]
we get an exact sequence 
\[0 \rightarrow Z(E) \xrightarrow{\eta} \Hom_E(M,M^{**}) \rightarrow \Hom_E(M,\coker(\theta)) \rightarrow 0.\]
As $\eta$ is an isomorphism, $\Hom_E(M,\coker(\theta))=0$. By a similar argument as above, we obtain that $\coker(\theta)$ is an $R$-module of finite length. Any ascending chain $C_1 \subseteq C_2 \subseteq \cdots \subseteq C_n$ of $E$-submodules $\coker(\theta)$ is also an ascending chain of $R$-modules, and as $\coker(\theta)$ has finite length as an $R$-modules, it follows that $\coker(\theta)$ also has finite length as a left $E$-module. Suppose $\coker(\alpha) \ne 0$ so that its $E$-module length is nonzero. Then as $\coker(\theta)$ has a finite composition series, it must contain a simple left $E$-module. But as $E$ is local, the only simple left $E$-module is $E/J(E)$, so there is an embedding $i:E/J(E) \hookrightarrow \coker(\theta)$. Moreover, as $E/J(E)$ is a division ring, $M/J(E)E$ is a free $E/J(E)$-module and so we take an $E$-linear projection $q:M/J(E)E \twoheadrightarrow E/J(E)$. Letting $p:M \to M/J(E)M$ be the natural projection of left $E$-modules, we see that $i \circ q \circ p$ is a nonzero element of $\Hom_E(M,\coker(\alpha))$. It follows that $\coker(\alpha)=0$. Therefore, $\theta$ is an isomorphism so $M$ is reflexive.


\end{proof}

\begin{remark}\label{exampleEndlocal}

The hypothesis that $\End_R(M)$ be local in Theorem \ref{Lindoconverse} (2) cannot be removed. Indeed, take for example $R$ to to be a non-Gorenstein domain of dimension $1$ with canonical module $\w_R$, e.g. one can take $R=k[\![t^3,t^4,t^5]\!]$. Set $M:=\w_R \oplus R$. Then $M$ is a projective left $E$-module, so $\Ext^1_E(M,M)=0$. Note also that $M$ is torsionless and free on the punctured spectrum since $R$ is a domain of dimension $1$. We have from Proposition \ref{basictraceprop} (3) that $R=Z(E) \cong Z(\End_R(M^*))=R$ since $M$ (and thus $M^*$) is a generator, but $M$ is not reflexive as $R$ is not Gorenstein. 
\end{remark}

\section{When is $Z(E) \cong R$?}\label{rigidsection}

In this section we provide several critera under which $Z(E) \cong R$. Using these, we provide partial results of several flavors towards the Huneke-Wiegand conjecture. Throughout this section, we let $R$ be a Cohen-Macaulay local ring of dimension $1$, and we let $M$ be a finitely generated $R$-module with constant rank. For ease of notation, we let $(-)^{\dagger \dagger}_E:=\Hom_E(\Hom_R(-,M),M)$.

We begin with the following result that describes how to calculate $\Omega^1_{Z(E)}(M)$ when $M$ is a rigid module.

\begin{prop}\label{firstprop}

Suppose $\Ext^1_R(M,M)=0$. Then 
\begin{enumerate}
\item[$(1)$] $\Omega^1_{Z(E)}(M) \cong (\Omega^1_R(M))^{\dagger \dagger}$ up to free $Z(E)$-summands, and
\item[$(2)$] $\Ext^1_{Z(E)}(M,M)=0$.
\end{enumerate}    
\end{prop}

\begin{proof}

Let $x_1,\dots,x_n$ be a minimal $R$-generating set for $M$ and let $p:R^{\oplus n} \to M$ be given on standard basis vectors by $p(e_i)=x_i$. Applying $\Hom_R(-,M)$ to the short exact sequence \[0 \rightarrow \Omega^1_R(M) \rightarrow R^{\oplus n} \xrightarrow{p} M \rightarrow 0\]
we obtain a short exact sequence
\[0 \rightarrow E \xrightarrow{\Hom_R(p,M)} \Hom_R(R^{\oplus n},M) \rightarrow \Hom_R(\Omega^1_R(M),M) \rightarrow 0\]
Applying $\Hom_E(-,M)$ to this exact sequence, we obtain an exact sequence
\[0 \rightarrow (\Omega^1_R(M))^{\dagger \dagger} \rightarrow (R^{\oplus n})^{\dagger \dagger} \xrightarrow{p^{\dagger \dagger}} \Hom_E(E,M).\]
Noting that $x_1,\dots,x_n$ is a $Z(E)$-generating set for $M$, we have a surjection $q:Z(E)^{\oplus n} \to M$ given on standard basis vectors by $q(e_i)=x_i$. 
Moreover, there is a commutative diagram:

\[\begin{tikzcd}[cramped,row sep=3.15em,yshift=.3ex]
	0 & {\ker(q)} & {Z(E)^{\oplus n}} & M & 0 \\
	0 & {(\Omega^1_R(M))^{\dagger \dagger}} & {(R^{\oplus n})^{\dagger \dagger}} & {\Hom_E(E,M)}
	\arrow[from=1-1, to=1-2]
	\arrow[from=1-2, to=1-3]
	\arrow[from=1-4, to=1-5]
	\arrow["q", from=1-3, to=1-4]
	\arrow["{p^{\dagger \dagger}}", from=2-3, to=2-4]
	\arrow["{i^{\dagger \dagger}}", from=2-2, to=2-3]
	\arrow["{\gamma^M_M}", from=1-4, to=2-4]
	\arrow["{\beta_M^{R^{\oplus n}}}", from=1-3, to=2-3]
	\arrow["w", dashed, from=1-2, to=2-2]
	\arrow[from=2-1, to=2-2]
\end{tikzcd}\]

As the maps $\beta^{R^{\oplus n}}_M$ and $\beta^M_M$ are isomorphisms, it follows that $p^{\dagger \dagger} $ is surjective and then that $w$ is an isomorphism. So we have established $(1)$. For $(2)$, as $\Ext^1_R(M,M)=0$, it follows from \cite[Lemma 5.3]{DE21} that $M^{\vee} \otimes_R M$ is torsion-free. Since $R$ is a subring of $Z(E)$, there is a surjection $\beta:M^{\vee} \otimes_R M \twoheadrightarrow M^{\vee} \otimes_{Z(E)} M$ whose kernel is torsion since $Z(E)_{\p} \cong R_{\p}$ for all $\p \in \Ass(R)$. As $M^{\vee} \otimes_R M$ is torsion-free, it follows $\beta$ is an isomorphism. Then
\[M^{\vee} \otimes_R M \cong M^{\vee} \otimes_{Z(E)} M \cong \Hom_R(M,\tr_{\w}(M)) \otimes_{Z(E)} M\]
\[\cong \Hom_{Z(E)}(M,\tr_{\w}(M)) \otimes_{Z(E)} M \cong \Hom_{Z(E)}(M,\w_{Z(E)}) \otimes_{Z(E)} M.\]
As $\Hom_{Z(E)}(M,\w_{Z(E)}) \otimes_{Z(E)} M$ is torsion-free, \cite[Lemma 5.3]{DE21} forces $\Ext^1_{Z(E)}(M,M)=0$ and we have $(2)$.


\end{proof}

\begin{theorem}\label{thm4.2}
Suppose $I$ is an ideal of $R$ of positive grade. If $\Ext^1_R(I,I)=\Ext^1_{E}(I,E)=0$, then $I$ is principal.
\end{theorem}

\begin{proof}
We first note that is it well-known that $E:=\End_R(I)$ is commutative in this situation; see e.g. \cite[Exercise 4.1.3]{Lw12}.
From \cite[Lemma 3.4 (1)]{LM20} (cf. \cite[Lemma 3.5 (2)]{KO22}), the condition that $\Ext^1_R(I,I)=0$ forces $I^{\vee} \otimes_R I$ to be torsion-free. Then from Proposition \ref{RvsSprop} (2) it follows that $I^{\vee} \otimes_R I \cong I^{\vee} \otimes_{E} I$. Moreover, the trace map $I^{\vee} \otimes_{E} I \twoheadrightarrow \tr_{\w_R}(I)$ also has a torsion kernel, and so $I^{\vee} \otimes_E I \cong \tr_{\w_R}(I)$. From \cite[Corollary 3.6]{Ly23}, we have $\w_E \cong \tr_{\w_R}(I)$, and then, similar to above, the condition $\Ext^1_{E}(I,E)=0$ forces $I \otimes_E \tr_{\w_R}(I)$ to be torsion-free via \cite[Lemma 3.4 (1)]{LM20}. But then
\[I \otimes_E \tr_{\w_R}(I) \cong I \otimes_E (I^{\vee} \otimes_R I) \cong (I \otimes_E I^{\vee}) \otimes_R I \cong (I^{\vee} \otimes_E I) \otimes_R I \cong (I^{\vee} \otimes_R I) \otimes_R I \cong I^{\vee} \otimes_R (I \otimes_R I).\]
Applying $I^{\vee} \otimes_R -$ to the natural exact sequence
\[0 \rightarrow \Tor^R_1(I,R/I) \rightarrow I \otimes_R I \xrightarrow{p} I^2 \rightarrow 0,\]
we have an exact sequence
\[I^{\vee} \otimes_R \Tor^R_1(I,R/I) \rightarrow I^{\vee} \otimes_R (I \otimes_R I) \xrightarrow{I^{\vee} \otimes p} I^{\vee} \otimes_R I^2 \rightarrow 0.\]
As $I^{\vee} \otimes_R \Tor^R_1(I,R/I)$ is torsion, so is $\ker(I^{\vee} \otimes p)$, and as $I^{\vee} \otimes_R (I \otimes I)$ is torsion-free from above, it follows that $I^{\vee} \otimes p$ is an isomorphism. In particular, comparing minimal number of generators on both sides, we have 
\[\mu_R(I^{\vee})\mu_R(I)^2=\mu_R(I^{\vee})\mu_R(I^2).\]
Since $I^{\vee}$ is nonzero, this forces $\mu_R(I)^2=\mu_R(I^2)$. But the map $p$ induces a surjection $S^2_R(I) \twoheadrightarrow I^2$, so in particular $\mu_R(I)^2=\mu_R(I^2) \le \dfrac{\mu_R(I)(\mu_R(I)+1)}{2}$. But this occurs if and only if $\mu_R(I) \le 1$, so $I$ is principal, as desired. 

\end{proof}

\begin{theorem}
Suppose $I$ is a two-generated ideal of positive grade. If $\Tor^R_1(I,E)=0$ then $E \cong R$. If in addition $I$ is reflexive, then $I$ is principal.
\end{theorem}

\begin{proof}

As $I$ is two-generated, there is an exact sequence
\[0 \to I^* \to R^{\oplus 2} \to I \to 0,\]
see \cite[Proof of Lemma 3.3]{HH05}. Applying $- \otimes_R E$, we have, since $\Tor^R_1(I,E)=0$, an exact sequence of the form
\[0 \to I^* \otimes_R E \to E^{\oplus 2} \to I \otimes_R E \to 0.\]
In particular, the depth lemma forces $I^* \otimes_R E$ to be torsion-free. But $I^*$ is naturally an $E$-module, and as $I$ is faithful, $R$ is naturally a subring of $E$. There is thus a surjection $I^* \otimes_R E \to I^* \otimes_E E \cong I^*$ of $R$-modules whose kernel is torsion by rank considerations. But as $I^* \otimes_R E$ is torsion-free, we have $I^* \otimes_R E \cong I^*$, which forces $\mu_R(E)=1$. As $E$ is a faithful $R$-module, it follows that $E \cong R$. If $I$ is in addition reflexive, then by Theorem \ref{Lindointro} and Proposition \ref{basictraceprop} we have $\End_R(\tr_R(I)) \cong (\tr_R(I))^* \cong R$, and it follows from \cite[Lemma 3.9]{DE21} that $\tr_R(I) \cong R$, so that $\tr_R(I)$ being a trace ideal must equal $R$. But then by Proposition \ref{basictraceprop} (3), it can only be that $I \cong R$, so $I$ is principal.
\end{proof}

\begin{theorem}\label{thm4.4}
Suppose $M$ is torsion-free and faithful, that $\Ext^1_R(M,M)=0$, and suppose additionally that one of the following holds:
\begin{enumerate}
\item[$(1)$] $M$ is reflexive as a $Z(E)$-module.
\item[$(2)$] $\Ext^2_{Z(E)}(M,M)=0$.

\end{enumerate}

Then $Z(E) \cong R$. If additionally, $M$ is reflexive as an $R$-module, then $M$ is a generator.

\end{theorem}

\begin{proof}


First suppose $M$ is a a reflexive $Z(E)$-module. By Theorem \ref{Lindointro}, Proposition \ref{RvsSprop} (3), and Proposition \ref{basictraceprop} (2), we have 
\[Z(\End_{Z(E)}(M))=Z(E) \cong \End_{Z(E)}(\tr_{Z(E)}(M)) \cong \Hom_{Z(E)}(\tr_{Z(E)}(M),Z(E)).\] Then \cite[Lemma 3.9]{DE21} forces $\tr_{Z(E)}(M) \cong Z(E)$, and being a trace ideal, we must have $\tr_{Z(E)}(M)=Z(E)$. It follows from Proposition \ref{basictraceprop} (3) that $M$ is a $Z(E)$-generator. But as $\Ext^1_R(M,M)=0$, this forces $\Ext^1_R(Z(E),Z(E))=0$, and then \cite[Lemma 3.4 (1)]{LM20} forces $Z(E) \otimes_R Z(E)^{\vee}$ to be torsion-free. But as $M$ is faithful, it follows from Proposition \ref{RvsSprop} (2) that there is a natural surjection $Z(E) \otimes_R Z(E)^{\vee} \twoheadrightarrow Z(E) \otimes_{Z(E)} Z(E)^{\vee} \cong Z(E)^{\vee}$ whose kernel is torsion. Since $Z(E) \otimes_R Z(E)^{\vee}$ is torsion-free, it follows $Z(E) \otimes_R Z(E)^{\vee} \cong Z(E)^{\vee}$, so that $\mu_R(Z(E))=1$. As $Z(E)$ is a faithful $R$-module, it follows that $Z(E) \cong R$. 

If additionally, $M$ is a reflexive $R$-module, then we have from Theorem \ref{Lindointro}, that $\End_R(\tr_R(M)) \cong Z(E) \cong R$, and then Proposition \ref{basictraceprop} (2) gives that $(\tr_R(M))^* \cong R$. Applying \cite[Lemma 3.9]{DE21} once more, we see that $\tr_R(M) \cong R$, and being a trace ideal, we have $\tr_R(M)=R$. That $M$ is a generator then follows from Proposition \ref{basictraceprop} (3), completing the proof of (1).

For (2), suppose $\Ext^2_{Z(E)}(M,M)=0 $. Let $x_1,x_2,\dots,x_n$ be a minimal generating set for $M$ and take a short exact sequence 
\[0 \rightarrow \Omega^1_R(M) \xrightarrow{i} R^{\oplus n} \xrightarrow{p} M \rightarrow 0\] with $p$ given on standard basis vectors by $p(e_i)=x_i$. As in the proof of Proposition \ref{firstprop}, there is an exact sequence

\[0 \rightarrow (\Omega^1_R(M))^{\dagger \dagger}  \xrightarrow{i^{\dagger \dagger}} (R^{\oplus n})^{\dagger \dagger} \xrightarrow{p^{\dagger \dagger}} M^{\dagger \dagger} \rightarrow 0\]
fitting in a commutative diagram:

\[\begin{tikzcd}[cramped,row sep=3.15em,yshift=.3ex]
	& {\Omega^1_R(M) \otimes_R Z(E)} & {R^{\oplus n} \otimes_R Z(E)} & {M \otimes_R Z(E)} & 0 \\
	0 & {(\Omega^1_R(M))^{\dagger \dagger}} & {(R^{\oplus n})^{\dagger \dagger}} & {M^{\dagger \dagger}} & 0
	\arrow["{i \otimes Z(E)}", from=1-2, to=1-3]
	\arrow[from=1-4, to=1-5]
	\arrow["{q \otimes Z(E)}", from=1-3, to=1-4]
	\arrow["{p^{\dagger \dagger}}", from=2-3, to=2-4]
	\arrow["{i^{\dagger \dagger}}", from=2-2, to=2-3]
	\arrow["{\beta^M_M}", from=1-4, to=2-4]
	\arrow["{\beta_M^{R^{\oplus n}}}", from=1-3, to=2-3]
	\arrow["{\beta_{M}^{\Omega^1_R(M)}}", from=1-2, to=2-2]
	\arrow[from=2-1, to=2-2]
	\arrow[from=2-4, to=2-5]
\end{tikzcd}\]

By the snake lemma, there is an exact sequence of the form
\[0 \rightarrow \Tor^R_1(M,Z(E)) \rightarrow \Omega^1_R(M) \otimes_R Z(E) \xrightarrow{\beta_M^{\Omega^1_R(M)}} (\Omega^1_R(M))^{\dagger \dagger} \rightarrow T \rightarrow 0\]
where $T:=\ker(\beta^M_M)$.
Splitting this exact sequence into short exact sequences 
\[\clubsuit: 0 \rightarrow \Tor^R_1(M,Z(E)) \rightarrow \Omega^1_R(M) \otimes_R Z(E) \xrightarrow{b} C \rightarrow 0\]

\[\spadesuit: 0 \rightarrow C \xrightarrow{a} (\Omega^1_R(M))^{\dagger \dagger} \rightarrow T \rightarrow 0\]

we see as $\Tor^R_1(M,Z(E))$ is torsion, that $\Hom_{Z(E)}(b,M)$ is an isomorphism. By Hom-tensor adjointness, we may identify $\Hom_{Z(E)}(q \otimes Z(E),M)$ and $\Hom_{Z(E)}(i \otimes Z(E),M)$ with $\Hom_R(p,M)$ and $\Hom_R(i,M)$ respectively. In particular, $\Hom_{Z(E)}(i \otimes Z(E),M)$ is surjective, and then so is $\Hom_{Z(E)}(\beta_M^{\Omega^1_R(M)},M)$ by commutativity of the diagram. As $\Hom_{Z(E)}(b,M)$ is an isomorphism, it follows that $\Hom_{Z(E)}(a,M)$ is surjective. 

Now, as $\Hom_E(\Hom_R(\Omega^1_R(M),M),M) \cong \Omega^1_{Z(E)}(M)$ uo to free $Z(E)$-summands, we have
\[\Ext^1_{Z(E)}(\Hom_E(\Hom_R(\Omega^1_R(M),M),M),M) \cong \Ext^1_{Z(E)}(\Omega^1_{Z(E)}(M),M) \cong \Ext^2_{Z(E)}(M,M)=0.\]

Thus, applying $\Hom_{Z(E)}(-,M)$ to $\spadesuit$ and considering the long exact sequence in $\Ext$, we see that $\Ext^1_{Z(E)}(T,M)=0$. But $T$ is torsion while $M$ is torsion-free, and as $Z(E)$ has dimension $1$, it can only be that $T=0$. Then $\beta^M_M$ is an isomorphism which forces $Z(E)$ to be a cyclic $R$-module. But as it is faithful, we have $Z(E) \cong R$. If $M$ is reflexive, then by \cite[Theorem 3.9]{Li17}, we have $\End_R(\tr_R(M)) \cong \Hom_R(\tr_R(M),R) \cong R$, so $\tr_R(M) \cong R$ by \cite[Lemma 2.13]{CG19}, which implies that $M$ is a generator.

\end{proof}


\section*{Acknowledgements}

This material is based upon work supported by the National Science Foundation under Grant No. DMS-1928930 and by the Alfred P. Sloan Foundation under grant G-2021-16778, while the second author was in residence at the Simons Laufer Mathematical Sciences Institute (formerly MSRI) in Berkeley, California, during the Spring 2024 semester.

\bibliographystyle{amsalpha}
\bibliography{mybib}

\end{document}